\documentclass[11pt,leqno]{article}
\usepackage[english]{babel}
\usepackage{tikz}
\usepackage{enumerate}
\usepackage{graphicx}
\usepackage{verbatim}
\usepackage{verbatim}
\usepackage{etoolbox}
\usepackage{tkz-euclide}
\usepackage{geometry}
\usepackage{mathtools}
\usepackage{color}
\usepackage[title]{appendix}
\usepackage[all]{xy}
\usepackage{tikz-cd}
\usepackage{blindtext}
\usepackage[T1]{fontenc}
\usepackage{amsthm}
\usepackage{amsfonts}
\usepackage{txfonts}
\usepackage{palatino,amssymb,epsfig}
\usepackage{latexsym,epsf,epic,amscd}
\usepackage{mathrsfs}
\usepackage{graphicx}
\usepackage{caption}
\usepackage{subcaption}
\usepackage{appendix}
\usepackage{amsmath}
\usepackage{bookmark} 
\usepackage[nottoc]{tocbibind}
\hypersetup{
colorlinks=true,
linkcolor=black,
citecolor=black,
anchorcolor=black,
urlcolor=black}
\allowdisplaybreaks[4]
\numberwithin{equation}{section}

\DeclareMathOperator{\area}{Area}
\DeclareMathOperator{\arccoth}{arccoth}
\DeclareMathOperator{\arctanh}{arctanh}
\DeclareMathOperator{\arccot}{arccot}
\newcommand{\ddt}[1]{\frac{\mathrm{d}#1}{\mathrm{d}t}}
\newcommand{\pp}[2]{\frac{\partial#1}{\partial#2}}
\newcommand{\D}{\mathrm{d}}
    \newcommand{\Addresses}{{
  \bigskip
  \footnotesize
  \noindent Guangming Hu, \href{18810692738@163.com}{18810692738@163.com}
\newline\textit{College of Science, Nanjing University of Posts and Telecommunications,
  Nanjing, 210003, P.R. China.}\par\nopagebreak
  \medskip
  \noindent
Yi Qi
\href{yiqi@buaa.edu.cn}{yiqi@buaa.edu.cn}
\newline\textit{ School of Mathematical Sciences,
Beihang University, Beijing, 100191, P.R. China}
 \par\nopagebreak
    \medskip
  \noindent Yu Sun, \href{yusun15185105160@163.com}{yusun15185105160@163.com}
  \newline\textit{School of mathematics and physics, Nanjing institute of technology, Nanjing, 211100, P.R. China.} \par\nopagebreak
    \medskip
  \noindent Puchun Zhou, \href{pczhou22@m.fudan.edu.cn}{pczhou22@m.fudan.edu.cn}
  \newline\textit {School of Mathematical Sciences, Fudan University, Shanghai, 200433, P.R. China} \par\nopagebreak
}}

\title{ Hyperbolic Circle Packings  and Total Geodesic Curvatures on  Surfaces with Boundary}
\author{ Guangming Hu, Yi Qi, Yu Sun and Puchun Zhou}

\date{}
\newtheorem{theorem}{Theorem}[section]
\newtheorem{lemma}[theorem]{Lemma}

\newtheorem{corollary}[theorem]{Corollary}
\theoremstyle{definition}
\newtheorem{definition}[theorem]{Definition}

\begin{document}
\maketitle

\begin{abstract}
 This paper investigates a  generalized hyperbolic circle packing (including circles, horocycles or hypercycles) with respect to the total geodesic curvatures on the surface with boundary.  We mainly focus on the existence and rigidity of circle packing whose contact graph is the $1$-skeleton of a finite polygonal cellular decomposition, which is analogous to the construction of Bobenko and Springborn \cite{bo}. Motivated by  Colin de Verdiere's method  \cite{colin},  we introduce the variational principle for 
generalized hyperbolic circle packings on polygons. By analyzing limit behaviours of generalized circle packings on polygons, we give an existence and rigidity for the generalized hyperbolic circle packing with conical singularities regarding the total geodesic curvature on each vertex of the contact graph. As a consequence, we introduce the combinatoral Ricci flow to  find  a desired circle packing with a prescribed total geodesic curvature on each vertex of the contact graph.

\end{abstract}

\section{Introduction}\label{sec1}
\subsection{Background}
The notion of  discrete conformal structures on  polyhedral surfaces is a discrete simulation corresponding to  conformal structures on  smooth surfaces. The research of circle packings on  polyhedral surfaces can be regarded as a  naturally  discrete problem of conformal structures. 
To study the hyperbolic structure on $3$-manifolds,
Thurston \cite{thurston}
 rediscovered   circle packings on triangulated closed surfaces, which are initiated
by  Koebe \cite{Koebe}  and Andreev\cite{andreev1, andreev2}. In the last few decades, many different extensions of circle packings on closed surfaces  have been widely concerned,  including circle patterns, generalized circle packings, vertex scaling, etc. For more details,  see \cite{bo, chow-Luo, David2, David3, guo-luo,Luo11, Luo2, Schlenker}  and others. 

The variational principle plays an important role to study  circle packings, which is  introduced by Colin de Verdiere \cite{colin}. There are nice works on variational principles
on polyhedral surfaces. We refer Bobenko and Springborn \cite{bo}, 
Leibon \cite{Leibon}, Rivin \cite{Rivin}, 
Luo \cite{Luo12,Luo13} and others.
Recently, Nie \cite{nie} discovered a new functional with respect to total geodesic curvatures by analyzing Colin de Verdiere's functional. As an application, Nie studied the rigidity of circle patterns in spherical background geometry and showed the existence and uniqueness of a circle pattern with spherical conical
metric for prescribing a total geodesic curvature at each vertex. The first author and
the third author of this paper studied in \cite{BHS} the existence and rigidity of generalized
hyperbolic circle packing metric with conical singularities on a triangulated surface regarding a total geodesic curvature on each vertex, where the generalized hyperbolic
circle packing may contain circles, horocycles and hypercycles.

 For the generalized hyperbolic circle packing,
one may ask the following two questions:

\begin{enumerate}
    \item Whether the results in \cite{BHS} can be generalized to a finite polygonal cellular decomposition or not.
    \item Whether the setting of \cite{BHS} can be extended to a surface with boundary or not.
\end{enumerate}
 
 In this paper, we give affirmative answers to the above two questions. We prove
that for a finite polygonal cellular decomposition of a surface with boundary and a
total geodesic curvature on each interior vertex and each boundary vertex, there exists
a unique generalized hyperbolic circle packing on the surface such that its contact graph
is the $1$-skeleton of the finite polygonal cellular decomposition.

In differential geometry, it is an important problem to  find a canonical metric with prescribed curvature on a given manifold. To solve this problem, the geometric flows are introduced and play a key role in the development of differential geometry. The Ricci flow was introduced by Hamilton \cite{Hamilton} and Perelman improved  Hamilton’s program of Ricci flow to solve the Poincar\'e conjecture and Thurston’s geometrization conjecture \cite{Perelman1, Perelman2, Perelman3}.
 
For the discrete geometry, the classical geometric flows can not work, since it loses the smooth property. Fortunately, Chow and Luo \cite{chow-Luo} introduced combinatorial Ricci flows to find Euclidean ( spherical or hyperbolic) polyhedral surfaces with zero discrete Gaussian curvatures on conical singularities. They also showed that the solution of the combinatorial Ricci flow exists for the long time and converges  under some combinatorial conditions in Euclidean and hyperbolic background geometry. There are many applications on combinatorial Ricci flows, for example, see \cite{Feng, Gehua, ge2, GeJiang, David, Gu1, Gu2, Gu3, Luo}.
 
 In this paper, we introduce the combinatoral Ricci flow to  find  a desired hyperbolic circle packing with a prescribed total geodesic curvature on each vertex of the contact graph.

\subsection{Main results}

\subsubsection{Finite polygonal cellular decomposition}

Let $S_{g,n}$ be a surface with boundary, which is obtained by removing $n$ topological open disks  from an oriented compact surface with genus $g\geq0$. Let $\eta :(V_{g,n},E_{g,n})\to S_{g,n} $ be an embedding of a graph $(V_{g,n},E_{g,n})$ into $S_{g,n}$ such that  there is at least one point of every
connected component of the boundary belongs to the vertex set $V_{g,n}$ and every connected component of the boundary is the union of some elements of the edge set  $E_{g,n}$.
Here we regard $V_{g,n}$ and $E_{g,n}$ as $\eta(V_{g,n})$ and $\eta(E_{g,n})$ respectively. A face is a connected
component outside of $V_{g,n}$ and $E_{g,n}$ on $S_{g,n}$. Denote by  $F_{g,n}$ the set of faces induced by $\eta.$ 

An embedding is called a \textbf{closed 2-cell embedding} if the following conditions hold:
\begin{enumerate}[1.]
    \item[(a)]The closure of each face is homeomorphic to a close disk.
    \item[(b)] Any face is bounded by a simple closed curve consisting of finite edges.
\end{enumerate}
Set $\Sigma_{g,n}$ a finite cellular decomposition of $S_{g,n}$ induced by a closed 2-cell embedding $\eta$, whose  1-skeleton is the graph $(V_{g,n}, E_{g,n})$.
Denote $V_{g,n},E_{g,n}$ and $F_{g,n}$ as $0$-cells, $1$-cells and $2$-cells of $\Sigma_{g,n}$, respectively. 

A finite cellular decomposition 
is called a \textbf{finite polygonal cellular decomposition}, if it satisfies the following conditions:
\begin{enumerate}
    \item[(I)] Each 2-cell is a polygon.
    \item[(II)]  Every 0-cell  meets at least three 1-cells.
    
   \item[(III)]  The 1-skeleton is a simple graph without loop or double edge.
\end{enumerate}

\subsubsection{Generalized circle packings on polygons}
This section is based on some basic results from the hyperbolic geometry.  The concepts of circle, horocycle and hypercycle are defined as follows, and for better understanding, see Figure \ref{cycles} on the Poincar\'{e} disk model $\mathbb{H}$.\begin{figure}[htbp]
\centering
\includegraphics[scale=0.26]{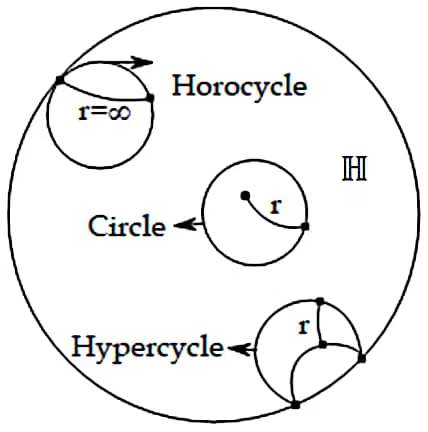}
\captionof{figure}{\small A circle, a hypercycle and a horocycle on $\mathbb{H}$}
  \label{cycles}
\end{figure} 
\begin{enumerate}[1.]
    \item[(1)] A hyperbolic \textbf{circle} $C(w,r)$ centered at $w\in \mathbb{H}$ with radius $r$ is defined as the subset
    $$C(w,r)=\{z\in \mathbb{H}:  d(w,z)=r \},$$
    which is also a Euclidean circle.
    \item[(2)] A hyperbolic \textbf{horocycle} 
    is the limit of a family of hyperbolic circles centered at points on a geodesic ray from a fixed point and passing though that fixed point. Indeed, it is a Euclidean circle in $\overline{\mathbb{H}}$
    which is tangent to the boundary of $\overline{\mathbb{H}}$. The
horocycle is also regarded as a hyperbolic circle centered at the ideal point on $\partial \mathbb{H}$ with radius $r=+\infty$. 
    \item[(3)] Given a geodesic $\gamma$ in $\mathbb{H}$, the \textbf{hypercycle} with axis $\gamma$ and radius $r$ refers to a connected component of the set $$C(\gamma,r)=\{z\in \mathbb{H}:d(\gamma,r)=r\},$$
where $d(z,\gamma)$ is the distance from $z$ to $\gamma$.

\end{enumerate}For simplicity,  circles, horocycles and hypercycles are collectively called \textbf{generalized circles}.
Generalized circles can be uniformly described as curves with constant geodesic
curvatures in  $\mathbb{H}.$

For the case of hypercycle,  
set $\alpha$  a sub-segment of $\gamma$ with two boundary points $b_{1}$ and $b_{2}$ and denote by $\beta_{1}$
and $\beta_{2}$ the  geodesic segments respectively passing $b_{1}$ and $b_{2}$ perpendicular to one component of $C(\gamma,r)$ having constant geodesic curvature.  We  set $C(\alpha,r)$ as the transverse segment of one component of $C(\gamma,r)$ and  call $\alpha$ 
the center of $C(\alpha,r)$. We can construct a half collar by adhering $\beta_{1}$ and $\beta_{2}$ as shown in  Figure \ref{3}.  The relation between their radii, absolute values of geodesic curvatures and arc lengths are
listed in Table $1$ and we refer to \cite{BHS} for details. 
In the case of hypercycle, $\theta$  represents the hyperbolic length of  center $\alpha$ of $C(\alpha,r)$.
Usually we call $\theta$ the \textbf{generalized inner  angle} of $C(\alpha,r)$. For the case of horocycle, the length of arc is calculated as Lemma $2.5$ in \cite{BHS}.

\begin{figure}[h]
\centering
\includegraphics[height = 2.6in, width=4.0in]{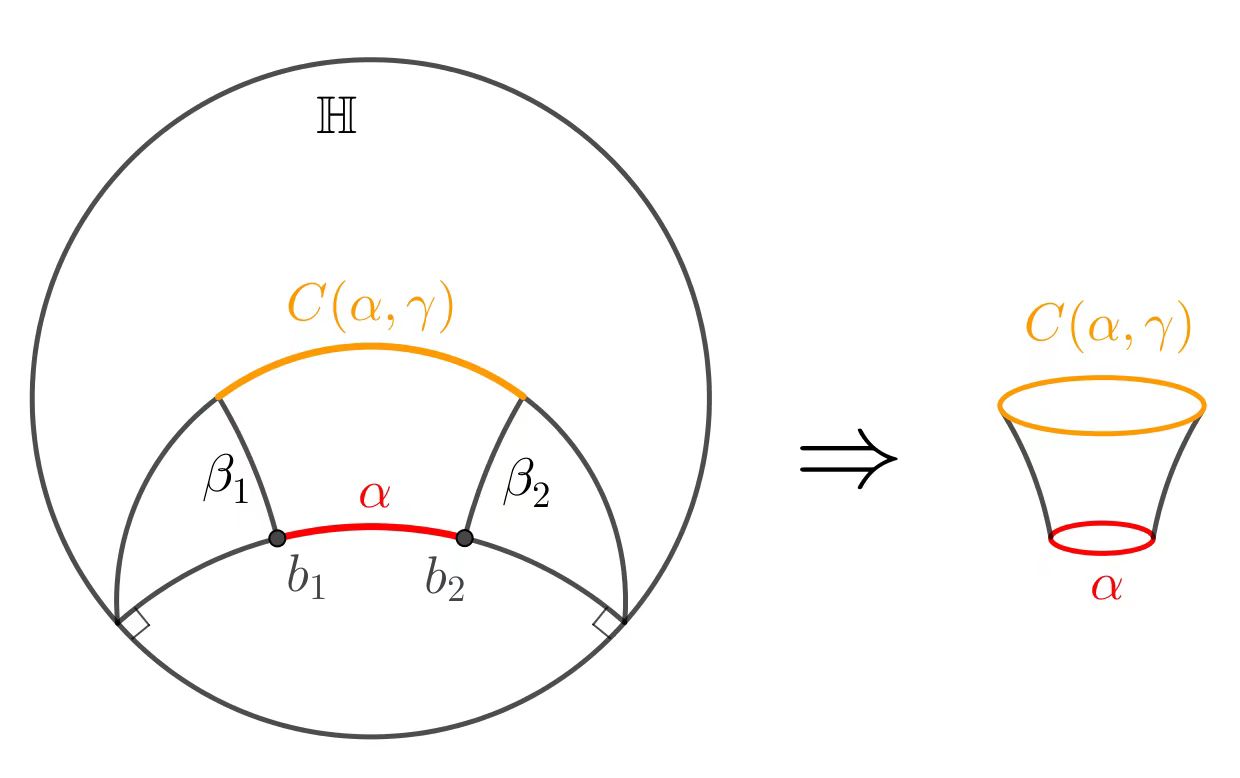}
\caption{Half collar}
\label{3}
\end{figure}
\begin{definition}[Generalized circle packing on polygon]\label{ccp_polygon}
Let $P$ be an abstract polygon with $n$ vertices  $v_1,\cdots,v_n.$  Let $C_i$ be a generalized circle in $\mathbb{H}$  corresponding to $v_i$. \textbf{A generalized circle packing } of $P$ is  the geometric pattern formed by $C_1,\cdots,C_n$  with the following conditions (see Figure \ref{generalcircle}):
\begin{enumerate}[i.]
    \item If $C_i$ and $C_j$ have a nonempty intersection, they must be tangent to each other.
    \item The generalized $C_i$ is tangent to $C_j$ if and only if $v_i$ and $v_j$ are connected by an edge of $P$.    
    \item There exists a hyperbolic circle $C_P$  which is perpendicular to all $C_i$ at tangent points.
\end{enumerate}
\end{definition}
\begin{table}\centering
\centering
\footnotesize	
\begin{tabular}
{l|l|l|l}
 &  \small Radius &  \small Absolute value of geodesic curvature &  \small  Arc length \\ \hline
 \small Circle &  $0<r<\infty$&  $k=\coth r$&  $l=\theta\sinh r$  \\ 
 \small Horocycle &  $r=\infty$&  $k=1$&  -----  \\
 \small Hypercycle &  $0<r<\infty$& $k=\tanh r$ &  $l=\theta\cosh r$  \\ \hline
 
\end{tabular}
\caption{\small The relationship of radius, absolute value of geodesic curvature and arc length}


\label{relation}
\end{table}

    In this definition, Conditions i. and ii. are similar to the circle packing on a triangulation in literature. However, in order to fix the geometry of a generalized circle packing on a surface with boundary with a finite polygonal cellular decomposition, Condition iii. is essentially
    necessary. For the case of geodesic curvatures, the following  Lemma plays an important role,  which will be proved in section 2.

    \begin{lemma}\label{packing on polygon}
        Given an abstract polygon $P$ with vertices $v_1,\cdots,v_n,$ let $k_i\in \mathbb{R}_+,i=1,\cdots,n$. Then there exists a generalized circle packing $\{C_{i}\}_{i=1}^{n}$ of $P$ with the geodesic curvature $k_i$ on each $C_i$. 
        Moreover, let $k_P$ be the geodesic curvature of the circle $C_P$ perpendicular to all $C_i$, then $k_P$ is a $C^1$ continuous function of $(k_1,\cdots,k_n).$
    \end{lemma}

Let $V_P$ be the vertex set of  abstract polygon $P$. Two vertices $v$, $w\in V_{P}$ are called
adjacent, denoted by $v\sim w$, if $v$ and $w$ are connected by an edge.
With the help of Lemma \ref{packing on polygon}, we construct a generalized hyperbolic geometric polygon $\Tilde{P}$ corresponding to the circle packing of abstract polygon. The generalized vertices of  $\Tilde{P}$ can be divided into three kinds as follows.
\begin{enumerate}[1.]
    \item The generalized vertex is an  interior point in $\mathbb{H}$ if the geodesic curvature $k_{i}$ defined on $v_{i}$  satisfies  $k_{i}>1$. 
    \item The generalized  vertex is  an  ideal point on the boundary of $\mathbb{H}$ if the geodesic curvature $k_{i}$ defined on $v_{i}$  satisfies $k_{i}=1$. 
    \item The generalized  vertex is a geodesic arc in $\mathbb{H}$ if the geodesic curvature $k_{i}$ defined on $v_{i}$  satisfies $k_{i}<1$. 
\end{enumerate}\begin{figure}[h]
		\centering
		\includegraphics[height = 2.3in, width=4.8in]{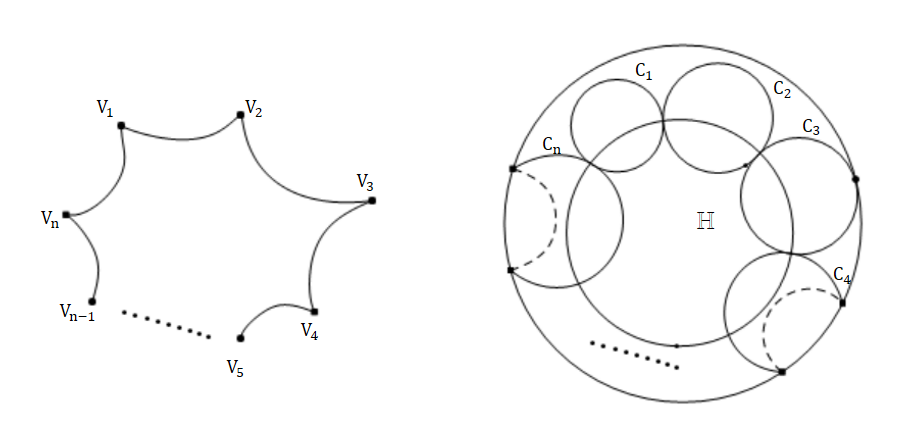}
		\caption{A generalized circle packing}
		\label{generalcircle}
	\end{figure}

Next, we give the geodesic length of  $\gamma_{ij}$  connecting $v_i$  and $v_j$ as the following four cases.
\begin{enumerate}[A.]
    \item If  $v_i $ and $v_j$ are two interior vertices in $\mathbb{H}$, then the hyperbolic distance between $v_i $ and $v_j$ is defined by $$d(v_i,v_j):=\arccoth k(v_i)+\arccoth k(v_j).$$ 
    \item If $v_i $ is an interior vertex in $\mathbb{H}$ and $v_j$ is a geodesic arc of $\mathbb{H}$, then the hyperbolic distance from $v_i $ to the geodesic arc $v_j$ is defined by
    $$
    d(v_i,v_j):=\arccoth k(v_i)+\arctanh k(v_j).
    $$ 
    \item If at least one of $v_i$ and $v_j$ is an ideal vertex  on the boundary of $\mathbb{H}$, then the hyperbolic distance is defined by $$d(v_i,v_j):=+\infty.$$
    \item If $v_i$ and $v_j$ are both geodesic arcs of $\mathbb{H}$, then the hyperbolic distance is defined by $$d(v_i,v_j):=\arctanh{k(v_i)}+\arctanh{k(v_j)}.$$
    \end{enumerate}
This (ideal) hyperbolic polygon $\tilde{P}$ is called the  \textbf{generalized polygon} of a \textbf{generalized circle packing} with geodesic curvatures $k=(k_{v_1},\cdots,k_{v_n})$. 
\begin{figure}[h]
\centering
\includegraphics[height = 2.5in, width=2.7in]{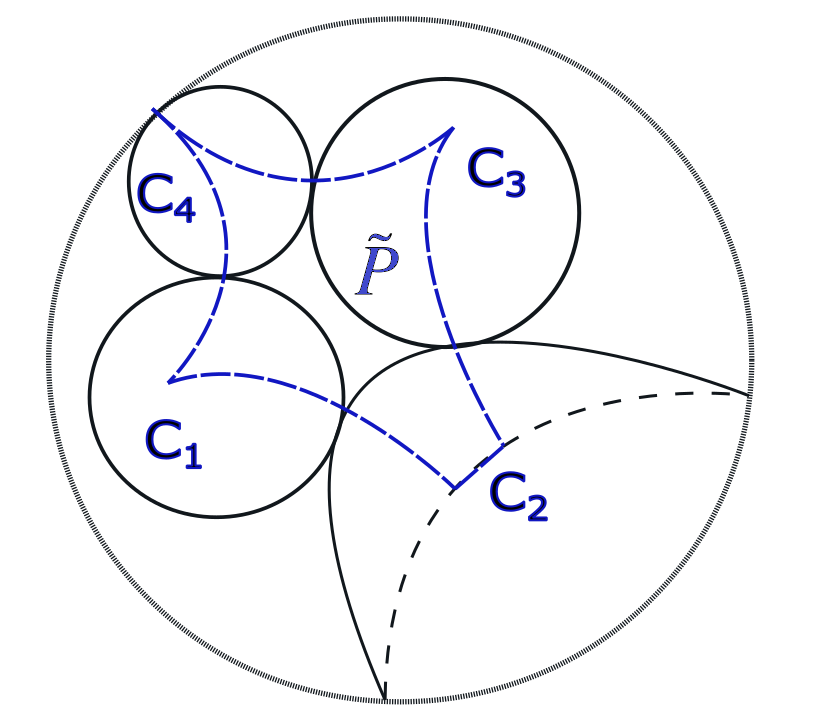}
\caption{A generalized quadrilateral of a generalized circle packing}
\label{hyperbolic_polygon}
\end{figure}

Figure \ref{hyperbolic_polygon} is shown a generalized  quadrilateral with $k_1,k_3>1, k_2<1$ and $k_4=1$.

\subsubsection{Generalized circle packings on surfaces with boundary}
Now we consider the generalized circle packing on surfaces with boundary in the hyperbolic background geometry. 
Set $S_{g,n}$ with a finite polygonal cellular decomposition $\Sigma_{g,n}$ and denote $V_{g,n}$, $E_{g,n}$ and $F_{g,n}$ as the sets of  0-cells, 1-cells and 2-cells, respectively. 
For simplicity of notations, we use one index to denote a vertex ($i\in V_{g,n}$), two indices to
denote an edge ($ij\in E_{g,n}$ is the arc on $S_{g,n}$ joining $i$, $j$) and $m$ indices to denote a face  ($P=i_{1}i_{2}\cdots i_{m}\in F_{g,n}$ is the
polygon  on $S_{g,n}$  bounded by $i_{1}i_{2},i_{2}i_{3}, \cdots,i_{m}i_{1}$).
\begin{figure}[h]
\centering
\includegraphics[height = 2.7in, width=3.5in]{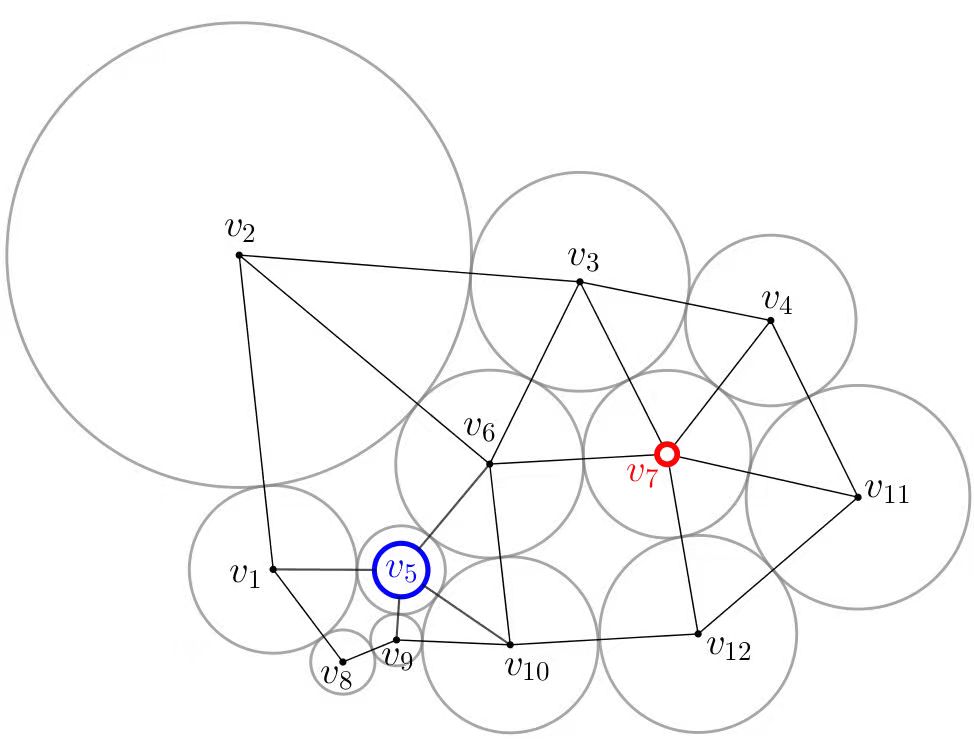}
\caption{A local representation of generalized circle packing}
\label{Local}
\end{figure}

\begin{definition}[Generalized hyperbolic circle packing]

Given  geodesic curvatures $k=(k_{1},\cdots,k_{\vert V_{g,n}\vert })\in \mathbb{R}_{+}^{\vert V_{g,n}\vert}$ on $V_{g,n}$, we can construct a hyperbolic surface $\tilde{S}=\tilde{S}(k)$ as follows:
\begin{enumerate}
    \item For each face $P\in F_{g,n}$ with $m$ vertices $i_{1},i_{2},\cdots, i_{m}\in V_{g,n}$, we construct a 
generalized  polygon $\tilde{P}$ of generalized circle packing  with geodesic curvatures $\textbf{k} = (k_{i_1}, \cdots, k_{i_m})$.
\item Gluing all the generalized  polygons together along their edges,  we can obtain a hyperbolic surface $\tilde{S}=\tilde{S}(k)$, whose metric is called a  \textbf{generalized circle packing metric}.
\end{enumerate}

\end{definition}
Let $V_1,V_2$ and $V_3$ be the set of vertices defined as follows.
\begin{enumerate}
    \item $i\in V_1$ iff $k_{i}\in(0,1)$.
    \item $i\in V_2$ iff $k_{i}=1$.
    \item $i\in V_3$ iff $k_{i}>1$.
\end{enumerate}
Topologically, $\tilde{S}=\tilde{S}(k)$ can be obtained by removing disjoint open disks (or half open disks) $\mathcal{B}_{V_{1}}$ containing $V_{1}$ and finite points $V_2$ from $S_{g,n}$.

  Set $\tilde{\mathcal{C}}=\{C_i\}_{i=1}^{|V_{g,n}|}$  the \textbf{generalized circle packing on $\tilde{S}$} with the geodesic curvatures $\textbf{k}=(k_1,\cdots,k_{|V_{g,n}|})$. Also, we set $\mathcal{D}=\{D_i\}_{i=1}^{|V_{g,n}|}$ the generalized disks bonded by $C_i.$ Figure \ref{Local} shows a local representation of generalized circle packing  $\hat{S}=\hat{S}(\textbf{k})$, where $v_{5}\in V_{1}$
and   $v_{7}\in V_2$.

By $V(P)$ we denote the vertices set of a face $P$.
Denote $N(P)$ as the number of the vertices of $P\in F_{g,n}$. For a subset $W\subset V_{g,n}$, let $F_W$  denote the set
\[
F_W = \{P\in F_{g,n}: \exists v\in W \text{ such that }v \text{ is a 
 vertex of } P \}.
\]
For $P\in F_W$, set $$N(P,W)=\sharp\{v\in W: v \text{ is a vertex of } P\}.$$ 
   The main goal of this paper is to provide the existence and rigidity for generalized hyperbolic circle packings on a surface with boundaries by the total geodesic curvature defined on each vertex. Next, we give the definition of the total geodesic curvature defined on each vertex.
\begin{definition}[Total geodesic curvature]
  The total geodesic curvature of an arc is the integral of geodesic curvature along the arc. The total geodesic curvature of the generalized circle packing  at each vertex is defined by the total geodesic curvature of the corresponding circle, or horocycle, or hypercycle at each vertex.
\end{definition}
 
The main theorem of this paper states as follows.   

\begin{theorem}\label{main1}Let $(S_{g,n},\Sigma_{g,n})$ be a surface with boundary, which  has a finite polygonal cellular decomposition. 
By $V_{g,n}$, $E_{g,n}$ and $F_{g,n}$ we denote the sets of vertices, edges and faces, respectively. Let $F_W$ be the set of faces having at least one vertex in $W$ for subset $W\subset V_{g,n}$. Then there exists a surface $\tilde{S}$ with  generalized circle packing metric defined as above having  total geodesic curvatures $(L_1,\cdots,L_{\vert V_{g,n}\vert} )\in\mathbb{R}^{\vert V_{g,n}\vert}_{+}$ on  vertices if and only if $(L_1,\cdots,L_{\vert V_{g,n}\vert})\in\mathcal{L}$, where 
\[\mathcal{L}=\left\{(L_1,\cdots,L_{\vert V_{g,n}\vert})\left\vert\right.  \sum_{w\in W}L_w<\sum_{P\in F_W}\pi \min\{N(P,W),N(P)-2\},~~\text{for each nonempty}\ W\subset
 V_{g,n}\right\}.\]
Moreover, the generalized hyperbolic circle packing  is unique if it exists. 
\end{theorem}

\subsubsection{Combinatorial Ricci flows}
It is meaningful to  study the generalized hyperbolic circle packing on surfaces with the prescribed total geodesic curvature on each vertex.
 Motivated by the combinatorial Ricci flow defined by Chow-Luo \cite{chow-Luo} to find a circle packing with prescribed discrete Gaussian curvatures,
 we introduce the combinatorial Ricci flow for finding a desired generalized circle packing with  prescribed total geodesic curvatures $\{\hat{L}_i\}_{i=1}^{\vert V_{g,n}\vert}$.

 \begin{definition}[Combinatorial Ricci flow]
\begin{align}
\ddt{k_i}=-(L_i-\hat{L}_i)k_i,~~\forall i\in V_{g,n}.\label{ricci_flow}
\end{align}
\end{definition}
For the combinatorial Ricci flow, we have the following theorem.

\begin{theorem}\label{flowthm}
    The flow \eqref{ricci_flow} converges for any given initial value if and only if $(\hat{L}_1,\cdots,\hat{L}_{\vert V_{g,n}\vert})\in\mathcal{L}$, where $\mathcal{L}$ is defined as Theorem \ref{main1}.
    Moreover, if it converges, it will converge exponentially fast to the desired circle packing metric.
\end{theorem}

The paper is organized as follows: In Section 2, we give the variational principle for generalized circle packings. In Section 3,  we derive Theorem \ref{main1} by analyzing the limit behaviour of polygons. In Section 4, we introduce some basic properties of the flow (\ref{ricci_flow}) and obtain Theorem \ref{flowthm}.

\section{Variational principle for generalized circle packing}\label{section2}

\subsection{Admissible space}
In this section, we will give the proof of Lemma \ref{packing on polygon}.
First, we introduce variational principles for a bigon formed by two generalized circle arcs $\gamma,\eta$ intersecting perpendicularly as shown in Figure \ref{Bigons}. The following lemma  is basic on the analysis in \cite{BHS}.

\begin{lemma}\label{vari_angle}
    Let $k_{\gamma},k_{\eta}$ be the geodesic curvatures of two generalized circle arcs $\gamma,\eta$ intersecting perpendicularly. By $\alpha$ and $\theta$ we denote the (generalized) inner angles of $\gamma$ and $\eta$. Moreover, $\alpha$ and $\theta$ can be viewed as functions of $k_{\eta}$ and $k_{\gamma}.$ Then for $k_{\eta}\in (1,\infty)$, \begin{align}\frac{\partial\theta}{\partial k_{\eta}}\label{d1}&=\frac{2k_{\eta}k_{\gamma}}{\sqrt{k_{\eta}^2-1}(k_{\gamma}^2+k_{\eta}^2-1)}>0,\\
    \pp{\theta}{k_{\gamma}}&=-\frac{\sqrt{k_{\eta}^2-1}}{k_{\gamma}^2+k_{\eta}^2-1}<0.\label{d2}
    \end{align}
    \
    Moreover, fixing $k_{\gamma}$, we have
    \begin{align*}
        \lim_{k_{\eta}\rightarrow 1^+}\theta(k_{\gamma},k_{\eta})=0,~~
        \lim_{k_{\eta}\rightarrow +\infty}\theta(k_{\gamma},k_{\eta})&=\pi.
    \end{align*}
\end{lemma}
\begin{proof} 
    We prove the Lemma by analyzing the generalized quadrilateral in two perpendicular circles as shown in Figure \ref{Bigons}. By $Q_{k_{\gamma},k_{\eta}}$ we denote those quadrilaterals.
    Set $r_{\eta}$  the radius of  $\eta$. Then we have $$k_{\eta}=\coth {r_{\eta}},$$ for $k_{\eta}>1$. There are three cases as follows, depicted as Figure \ref{Bigons} .
    \begin{enumerate}[1.]
\item
    Suppose $k_{\gamma}>1$, then $k_{\gamma}=\coth{r_{\gamma}}.$
   By the law of the hyperbolic triangle, we have
    \begin{align}\cot\frac{\theta}{2}=\coth r_{\gamma}\sinh r_{\eta}=\frac{k_{\gamma}}{\sqrt{k_{\eta}^2-1}}\label{theta}.\end{align} 
     Differentiating both sides of (\ref{theta}) with respect to $k_{\eta}$ and $k_{\gamma}$,  we  have
\begin{align*}
    -\frac{1}{2\sin^2\frac{\theta}{2}}\frac {\partial\theta}{\partial k_{\eta}}&=-\frac{k_{\eta}k_{\gamma}}{(k_{\eta}^2-1)^{\frac{3}{2}}}
, \\ -\frac{1}{2\sin^2\frac{\theta}{2}}\frac {\partial\theta}{\partial k_{\gamma}}&=\frac{1}{\sqrt{k_{\eta^2-1}}}.
 \end{align*}   
Therefore,
\begin{align*}
    \pp{\theta}{k_{\eta}}&=\frac{2k_{\eta}k_{\gamma}}{\sqrt{k_{\eta}^2-1}(k_{\gamma}^2+k_{\eta}^2-1)},\\
    \pp{\theta}{k_{\gamma}}&=-\frac{\sqrt{k_{\eta}^2-1}}{k_{\gamma}^2+k_{\eta}^2-1}.
\end{align*}
Then we know  
$\frac{\partial\theta}{\partial k_{\eta}}>0$, when $k_{\eta}>1.$
Moreover, by (\ref{theta}) we have
$$\cot{\frac{\theta}{2}}\rightarrow+\infty,~~\text{when}~k_{\eta}\rightarrow 1^+.$$ Hence $\theta\rightarrow 0$ as $k_{\eta}\rightarrow 1^+$.
Moreover, due to \eqref{theta}, we also have $  \theta\rightarrow \pi$ as $k_{\eta}\rightarrow+\infty.$ Thus we finish the case of $k_{\gamma}>1.$

\item Suppose $k_{\gamma}<1,$ then $k_{\gamma}=\tanh r_{\gamma}$.
By the law of the hyperbolic quadrilateral, we obtain
\begin{align}
\cot\frac{\theta}{2}=\tanh r_{\gamma}\sinh r_{\eta}=\frac{k_{\gamma}}{\sqrt{k_{\eta}^2-1}}\label{tt}.
\end{align}
Viewing $\theta$ as function of $k_{\gamma}$ and $k_{\eta}$, $\theta$ has the same expression as shown in \eqref{theta} and \eqref{tt}. Therefore, \eqref{d1} and \eqref{d2} also hold. 

    Moreover, by \eqref{tt} we have
            \[\lim_{k_{\eta}\rightarrow 1^+}\theta(k_{\gamma},k_{\eta})=0, \quad   \lim_{k_{\eta}\rightarrow +\infty}\theta(k_{\gamma},k_{\eta})=\pi.\]
  \item 
Suppose $k_{\gamma}=1.$ Then we can also obtain a semi-ideal hyperbolic quadrilateral. In this case, we have 
\begin{align}\label{tt2}
\cot{\frac{\theta}{2}}=\sinh{r_{\eta}}=\frac{1}{\sqrt{k_{\eta}^2-1}}. 
\end{align}
Therefore, similar to the arguments of other cases, \eqref{d1} and \eqref{d2} also hold. Moreover, by \eqref{tt2} we also have
            \[\lim_{k_{\eta}\rightarrow 1^+}\theta(k_{\gamma},k_{\eta})=0, \quad   \lim_{k_{\eta}\rightarrow +\infty}\theta(k_{\gamma},k_{\eta})=\pi.\]
  \end{enumerate}
Thus, we complete the proof of this 
 lemma.
\end{proof}

\begin{figure}[htbp]
\centering
\includegraphics[scale=0.40]{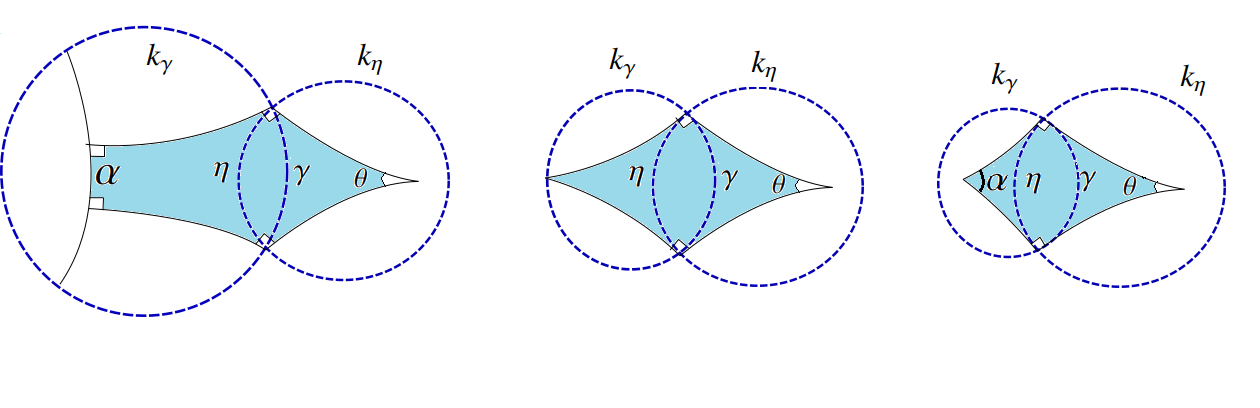}
\caption{ Generalized hyperbolic quadrilaterals in two perpendicular circles $Q_{k_{\gamma}k_{\eta}}$.}
\label{Bigons}
\end{figure}

Next, we will give  the proof of Lemma \ref{packing on polygon}.

\begin{proof}[Proof of Lemma \ref{packing on polygon}]
    
    Suppose that $P$ is an abstract polygon with $n$ vertices and $C_1,\cdots,C_n$ are generalized circles with geodesic curvatures $(k_1,\cdots,k_n)\in \mathbb{R}^{n}_{+}$. 
    For each $i$, we can construct a generalized hyperbolic quadrilateral $Q_{k_i,k'}$ , where  $k'\in(1,\infty)$ is the geodesic curvature of a circle $C'$. By $\theta_{i}=\theta_i(k_i,k')$ we denote the inner angle at the center of $C'$ in $Q_{k_i,k'}$.
    Gluing all the quadrilaterals together at the center of $C'$, we can get a hyperbolic cone metric with a cone angle $\theta(k_1,...,k_n;k'):=\sum_{i=1}^{n}\theta_{i}(k_i,k').$
    For finding a generalized circle packing $\{C_i\}_{i=1}^n$ of $P$ with geodesic curvatures $\{k_i\}_{i=1}^n$, we need to show that there exists some $k'$ such that
    \[
    \sum_i\theta_{i}(k_i,k')=2\pi.
    \]
    
     By Lemma \ref{vari_angle},  the cone angles varies from $0$ to $n\pi$ continuously as $k'$ varies in $(1,+\infty)$. And $\theta$  monotonically increases with respect to $k'$. For $n>2$, there exists a unique $k'=k_P$ such that the cone angle equals to $2\pi$. Thus we obtain a generalized circle packing of  $P$ defined in Section \ref{sec1}.
    
    Owing to Lemma \ref{vari_angle}, $\pp{\theta}{k'}>0$. Moreover, we have 
    $$\pp{\theta}{k_i}=-\frac{2\sqrt{(k')^2-1}}{k^{2}_{i}+(k')^{2}-1}<0.$$ Therefore, for each $i$, by the implicit function theorem, $k_P$ is a $C^1$ function of $(k_1,\cdots,k_n).$ Hence we finish the proof of Lemma \ref{packing on polygon}. 
    \end{proof}

\subsection{Convex functionals}
Now we consider the variational principle for total geodesic curvatures. The following lemma can be derived by a direct computation. For the  proof, we refer readers to \cite{BHS}.

\begin{lemma}\label{bigon}
   Let $\gamma$, $\eta$ be two generalized circle arcs to intersect each other perpendicularly, whose geodesic curvatures are  $k_{\gamma}$ and 
   $k_{\eta}$. Denote $l_{\gamma}$ and  $l_{\eta}$ as their corresponding  arc lengths. If $k_{\eta}>1$, then the differential form \[\omega_{\gamma,\eta}=l_{\gamma}dk_{\gamma}+l_{\eta}dk_{\eta}\]
    is closed.
\end{lemma}

With the help of Lemma \ref{bigon}, we have the following closed form for the  polygon of generalized circle packing metric.
\begin{lemma}\label{closed}
  Let $\tilde{P}$ 
 be the generalized polygon of generalized circle packing   with respect to  an abstract polygon $P$,
  where the geodesic curvatures of vertices $v_1,\cdots,v_n$ are $k_1,\cdots,k_{n}\in \mathbb{R}_+.$
  Let $\{C_i\}_{i=1}^n$ be the  generalized circle of the 
geodesic curvature $k_i$ and denote by $l_{i,P}$  the  length of the sub-arc of $C_i\cap\tilde{P}$.  Then the differential form
    \[
    \omega_{\tilde{P}}=\sum_{i=1}^{n}l_{i,P}\mathrm{d}k_i
    \]
     is a closed form.
\end{lemma}

\label{Bi}

\begin{proof}

Due to Lemma \ref{packing on polygon}, there exists 
   a circle $C_P$ intersecting to $C_i$ $(i=1,\cdots, n)$   perpendicularly, whose geodesic curvature $k_P$ is a function of $(k_1,\cdots,k_n)\in\mathbb{R}^n_+.$ Let $l_{C_{P}}$ be the circumference of $C_P$, which is a $C^1$ function of $k_P.$
    Then the form
    \[\omega_{C_P}=l_{C_{P}}dk_P\]
    is a closed form.

    Set $l_{C_P}^{i}$  the length of $C_P$ contained in the interior of $C_i$, depicted in Figure \ref{generalcircle}. Then by Lemma \ref{bigon}, the form
    \[l_{i,P}\mathrm{d}{k_i}+l_{C_{P}}^{i}\mathrm{d}k_P\]
    is  closed. Since $l_{C_{P}}=\sum_{i=1}^{n}l^{i}_{C_{P}}$, we obtain that the form 
    \[\omega_{\tilde{P}}+\omega_{C_P}=\sum_{i=1}^{n}(l_{i,P}\mathrm{d}k_i+l^{i}_{C_{P}}\D k_P). \]
    is a closed. 
    Therefore, $\omega_{\tilde{P}}$ is a closed form.
\end{proof}
  Denote  $L_{i,P}$ as the total geodesic curvature of the sub-arc $C_i\cap \tilde{P}$.
According to the definition of   the total geodesic curvature, we know that $L_{i,P}=l_{i,P}\cdot k_{i}$. Thus  $L_{i,P}$ is  smooth with respect to  $(k_1,\cdots ,k_n)$.
Set $s_i=\ln k_i$. Since 
$\sum_{i=1}^nl_{i,P}\D k_{i}$ is closed, the form 
$\omega=\sum_{i=1}^n L_{i,P} \D s_{i}$
is also closed.
Hence  the potential function  
\[\mathcal{E}_P(s)=\int^{s}_{s_0}\omega\]
 is well-defined on $\mathbb{R}^{n}$.

\begin{lemma}\label{gauss}
    Let $C_1,\cdots,C_n$ be a generalized circle packing of an abstract polygon $P$. Let $\Omega_P$ be the region enclosed by those arcs among tangent points. Then 
    \[
    \area(\Omega_P)=(n-2)\pi -\sum_{i=1}^{n}L_{i,P}.
    \]
\end{lemma}
\begin{proof}
    This lemma is directly obtained by the Gauss-Bonnet theorem.
\end{proof}
Therefore, $\sum_{i=1}^{n}L_{i,P}$ has a geometric interpretation closely related to the area of $\Omega_P.$
Now we come to study $\sum_{i=1}^{n}L_{i,P}$ as a function of $(k_1,\cdots,k_n).$ With the help of Table \ref{relation}, a simple calculation
gives
\begin{align}\label{expression}
L_{i,P}=\left\{
\begin{array}{lc}
\alpha_{i,P}\frac{k_i}{\sqrt{k_i^2-1}},&k_i>1, \\l_{i,P},&k_i=1,\\ \alpha_{i,P}\frac{k_i}{\sqrt{1-k_i^2}},&0<k_i<1,
\end{array}
\right.
\end{align}
where $\alpha_{i,P}$ is the (generalized) inner angle of the arc $C_i\cap\tilde{P}$ and $l_{i,P}$ is its length as in Lemma \ref{closed}.
The explicit expressions of $\alpha_{i,P}$ is given by \cite[Lemma 2.6]{BHS}. Here, we only need to use the following result.
\begin{lemma}
[\cite{BHS}]\label{angle_expression}
    The angle $\alpha_{i,P}$ is a $C^1$-function of $k_i$ and $k_P$. Moreover, we have
    \begin{align}\label{angle_compute}
\alpha_{i,P}(k_i,k_P)=\left\{
\begin{array}{lc}
2\arccot{\frac{k_P}{\sqrt{k_i^2-1}}},&k_i>1, \\ 2\arccoth\frac{k_P}{\sqrt{1-k_i^2}},&0<k_i<1,
\end{array}
\right.
\end{align}
\end{lemma}
Then we have the following Lemma.
\begin{lemma}\label{area_change}
Let $C_1,\cdots,C_n$ be a generalized circle packing of an abstract polygon $P$ with geodesic curvatures $(k_1,\cdots,k_n)\in \mathbb{R}_{+}^{n}.$ Let $\Omega_P$ be defined as above. Then 
\[
\pp{(\sum_{j=1}^{n}L_{j,P})}{k_i}>0, ~~ 1\leq i\leq n.
\]
As a consequence, 
\[
\pp{\area{(\Omega_P})}{k_i}<0, ~~ 1\leq i\leq n.
\]
\end{lemma}
    \begin{proof}
We divide $\Omega_P$ into n parts by the geodesic segments  connecting those tangency points of $C_{1},\cdots, C_{n}$ to the center of $C_P$ as shown in Figure \ref{3a}. Each part is a generalized hyperbolic quadrilateral as shown in Figure \ref{Bigons}. Denote $S_j$ $(j\neq i)$ as the triangle  whose boundary contain an arc of $C_j$ and the corresponding geodesic segments joining two tangency points of $C_{j}$ to the center of $C_P$. By $\theta_{j,P}$ we denote the inner angle of the arc $S_j\cap C_P$.
For simplicity, we define two auxiliary functions $f$ and $g$.
\[
f(x,y)=\left\{
\begin{array}{lc}
\frac{2x}{\sqrt{x^2-1}}\arccot{\frac{y}{\sqrt{x^2-1}}},&x>1, \\\frac{1}{y},&x=1,\\ 2\frac{x}{\sqrt{1-x^2}}\arccoth{\frac{y}{1-x^2}},&0<x<1,
\end{array}
\right.
\]
and 
\[
g(x,y)=2\arccot{\frac{x}{\sqrt{y^2-1}}}.
\]
$f$ and $g$ are functions of total geodesic curvature and general angle of two perpendicular arcs as in Lemma \ref{vari_angle} and \ref{angle_expression}.
And we can write 
\[
L_{j,P}=f(k_j,k_P(k_1,...,k_n))
.\]

 Naturally, we can get the following partial derivative  
\begin{align}
\pp{(\sum_{j=1}^{n}L_{j,P})}{k_i}&=\sum_{j=1}^{n}f_y(k_j,k_P)\pp{k_P}{k_i}+f_x(k_i,k_P).\label{total_vari}
\end{align}
 By simply computation, we have 
\begin{align*}
f_y(k_j,k_P)=2\frac{k_{j}}{1-k_{P}^{2}-k_{j}^{2}},~~\forall~~ 1\le j\le n.
\end{align*} 
By Lemma \ref{vari_angle} and the implicit function theorem, it has
\[
\pp{k_P}{k_i}=-\frac{g_x(k_i,k_P)}{\sum_jg_y(k_j,k_P)}.
\]
Since $$g_x(k_i,k_P)=\frac{2}{1-k^{2}_{i}-k^{2}_{P}}\sqrt{k_P^2-1}$$ 
and
\begin{align*}
g_y(k_j,k_P)=2\frac{k_{P}k_{j}}{k_{P}^{2}+k_{j}^{2}-1}\cdot \frac{1}{\sqrt{k_{P}^{2}-1}},
\end{align*}
Therefore, we have
\begin{align}
    \sum_{j=1}^{n}f_y(k_j,k_P)\pp{k_P}{k_i}&=-2\sum_{j=1}^{n}\frac{k_j}{1-k_P^2-k_j^2}\cdot\frac{2\frac{\sqrt{k^2_P-1}}{1-k_P^2-k_i^2}}{2\sum_{j=1}^{n}\frac{k_Pk_j}{(k_P^2+k_j^2-1)\sqrt{k_P^2-1}}}\nonumber\\
    &=\frac{2(k_P^2-1)}{k_P(1-k_P^2-k_i^2)}.\label{area1.1}
\end{align}
Now we consider $f_x(k_i,k_P)$.
We calculate it in three cases.
\begin{enumerate}[1.]
    \item We first consider the case as $k_i>1.$
Since 
\begin{align}
\label{area1.2}
f_x(k_i,k_P)=\frac{2k_i^2k_P}{(k_P^2+k_i^2-1)(k_i^2-1)}-\alpha_{i,P}(k_i^2-1)^{-\frac{3}{2}},
\end{align}
by \eqref{area1.1} and \eqref{area1.2}, we have
\[
\pp{(\sum_{j=1}^{n}L_{j,P})}{k_i}=-\alpha_{i,P}(k_i^2-1)^{-\frac{3}{2}}+\frac{2}{(k_i^2-1)k_P}
.\]
Since 
\[
\alpha_{i,P}=2\arctan{\frac{\sqrt{k_i^2-1}}{k_P}}
<\frac{2\sqrt{k_i^2-1}}{k_P},\]
we have 
\[
\pp{(\sum_{j=1}^{n}L_{j,P})}{k_i}>0.
\]
\item Now we consider the case as $0<k_i<1.$
By similar computation, we have 
\begin{align}\label{area2.1}
f_x(k_i,k_P)=\frac{-2k_i^2k_P}{(k_P^2+k_i^2-1)(1-k_i^2)}+\alpha_{i,P}(1-k_i^2)^{-\frac{3}{2}}.
\end{align}
Similarly, by \eqref{area1.1} and \eqref{area2.1}, we obtain
\[
\pp{(\sum_{j=1}^{n}L_{j,P})}{k_i}=-\frac{2}{k_P(1-k_i^2)}+\alpha_{i,P}(1-k_i^2)^{-\frac{3}{2}}.
\]
Since $\alpha_{i,P}=2\arctanh\frac{\sqrt{1-k_i^2}}{k_P}>2\frac{\sqrt{1-k_i^2}}{k_P},$
we have
\[
\pp{(\sum_{j=1}^{n}L_{j,P})}{k_i}>0.
\]
\item 
Finally, we consider the case as $k_i=1$. Since $L_{j,P}$ is $C^1$-function of $(k_1,...,k_n)$, we can take the limit 
\[
\lim_{k_i\rightarrow1^+}\pp{(\sum_{j=1}^{n}L_{j,P})}{k_i}=\frac{2}{3k_P^3}>0.
\]

By Lemma \ref{gauss}, we have 
\[
\pp{\area{(\Omega_P})}{k_i}<0, ~~ 1\leq i\leq n.
\]
Thus we finish  the proof completely.
\end{enumerate}
\end{proof}
\begin{figure}
     \centering
     \begin{subfigure}[c]
     {0.38\textwidth}
     \captionsetup{justification=centering}
         \centering
         \includegraphics[width=0.99\textwidth]{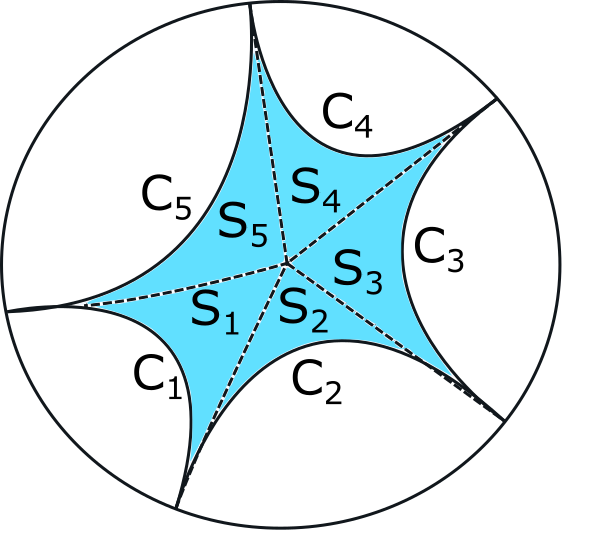}
         \caption{\small  A decomposition of $\Omega_{P}$}
         \label{3a}
     \end{subfigure}
     \begin{subfigure}[c]{0.43\textwidth}
\captionsetup{justification=centering}
         \centering
\includegraphics[width=1.50\textwidth]{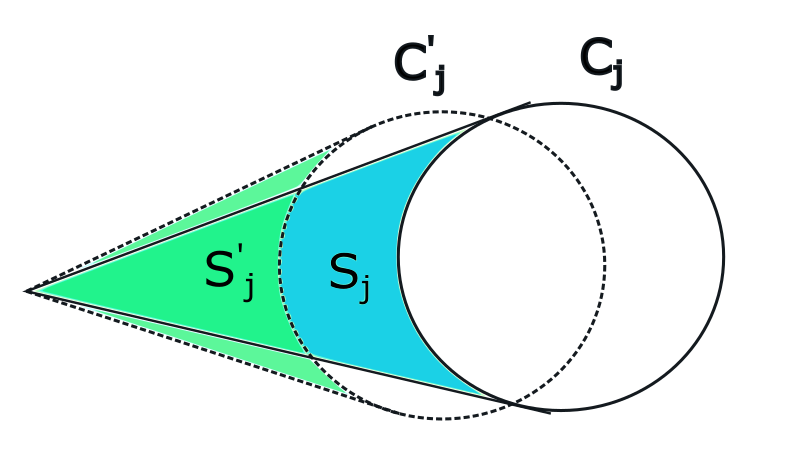}
        \caption{\small A deformation of  $S_{j}, \text{ for } j\neq i$ }
         \label{3b}
     \end{subfigure}
         \caption{\small A deformation of $\Omega_{P}$}
        \label{abcde}
\end{figure}

\begin{lemma}\label{convex}
The potential function $\mathcal{E}_P(s)$ is strictly convex.
\end{lemma}
\begin{proof}
The Hessian matrix of $\mathcal{E}_{P}(s)$ is
\[\mathrm{M}_{P}=\left[\frac{\partial L_{i,P}}{\partial s_j}\right]_{n\times n}.\]
From the proof of Lemma \ref{packing on polygon}, we can deduce that $\pp{k_P}{k_i}>0$ for$1\leq i\leq n$. Namely, fixed $k_j$ for any $j\neq i$, the radius $r_P$  of circle $C_P$ decreases as $k_i$ increases, as shown in Figure \ref{3b}.
Therefore the length $l_{j,P}$ will decrease as $r_P$ decreases. Since
$L_{j,P}=k_jl_{j,P}$, we have $\frac{\partial L_{j,P}}{\partial k_i}<0$.  Moreover, we know $\frac{\partial L_{j,P}}{\partial s_i}<0$ by $s_i=\ln k_i$.

By Lemma \ref{gauss}, we obtain
\[\frac{\partial L_{i,P}}{\partial s_i}=-\frac{\partial(\area(\Omega_{P})+\sum_{j\neq i}L_{j,P})}{\partial s_i}>0.\]
Meanwhile, by Lemma \ref{area_change}, 
we have
\[\left\vert\frac{\partial L_{i,P}}{\partial s_i}\right\vert-\sum_{j\neq i}\left\vert\pp{L_{j,P}}{s_i}\right\vert=
\frac{\partial(\sum_{k=1}^nL_{k,P})}{\partial s_i}>0.\]
Therefore, the Hessian matrix of $\mathcal{E}_{P}(x)$ is a symmetric strictly diagonally dominant matrix with positive diagonal entries, which yields that $\mathrm{M}_{P}$ is positive definite. Thus, the potential function $\mathcal{E}_P(s)$ is strictly convex.
\end{proof}

\section{The proof of main theorem}
\subsection{Limit behaviour}
Now we consider the limit behaviour of a generalized circle packing on an abstract polygon $P$. Set $\{k^m=(k_1^m,\cdots,k_n^m)\}_{m=1}^{+\infty}\subset\mathbb{R}_+^{n}$ a sequence of geodesic curvatures.
Then we obtain a sequence of polygons $\{\tilde{P}^m\}_{m=1}^{+\infty}$ with generalized circle packings that have circles with geodesic curvatures $\{k^m\}_{m=1}^{+\infty}.$ 
Denote $C_i^m$ and  $C_{P^m}$ as the corresponding circle whose geodesic curvatures are $k_i^m$ and $k_{P^m}$, respectively.
 Let $L_{i,P^m}$ and $l_{i,P^m}$ be the total curvature and length of the arc $C_i^m\cap \tilde{P}^m$. 
 
 Set $c=(c_1,\cdots,c_n)\in [0,+\infty]^n$  a non-negative (possible infinity) vector. Consider the limit behavior of $\{\tilde{P}^m\}_{m=1}^{\infty}$ as 
 \[
\lim_{m\rightarrow+\infty}k^m\rightarrow c=(c_1,\cdots,c_n),
\]
then we have the following lemma.
\begin{lemma}\label{limita0}
    If $c_j=0$ for each $1\leq j\leq n$, then 
$\lim_{m\rightarrow+\infty}L^m_{j,P}=0$ for each $j=1,\cdots,n$.
    
\end{lemma}
\begin{proof}
We divide our proof into two steps.
     
     First, we consider a special case of this lemma. Assume that 
    \[k_j^m=\frac{1}{m},~~\text{for each }  1\leq j\leq n,
    \] Since $k_j^m<1$ and $\theta_{j,P^m}=\frac{2\pi}{n}$ for each $j$, by \eqref{tt} we have
    \begin{align}
    \label{equal}
    k^j_{P^m}=\sqrt{1+\frac{\tan^2{\frac{\pi}{n}}}{m^2}},~~\forall j=1,\cdots,n.
    \end{align}
    Then by Lemma \ref{angle_expression} and \eqref{equal}, we have
    \begin{align*}
    \alpha_{j,P^m}&=2
    \arccoth\sqrt{\frac{1+\frac{\tan^2{\frac{\pi}{n}}}{m^2}}{1-\frac{1}{m^2}}},\nonumber
    \\&=\ln\frac{\sqrt{1+\frac{\tan^2\frac{\pi}{n}}{m^2}}+\sqrt{1-\frac{1}{m^2}}}{\sqrt{1+\frac{\tan^2\frac{\pi}{n}}{m^2}}-\sqrt{1-\frac{1}{m^2}}}.
    \end{align*}
   There exists a  constant $D>0$  such that $\alpha_{j,P^m}\le D-\ln\bigg(\sqrt{1+\frac{\tan^2{\frac{\pi}{n}}}{m^2}}-1\bigg)$ as $m$ is sufficiently large. Therefore, by \eqref{expression} 
    we have
    \[
    L_{j,P^m}=\frac{1}{\sqrt{m^2-1}}\alpha_{j,P^m}\rightarrow 0~~\text{as}~~m\rightarrow+\infty.
    \]
    Next, we consider the generalized case. Since total geodesic curvatures are positive, we only need to prove 
    \begin{align}\label{limsup}
        \limsup_{l\to+\infty}{\sum_{j=1}^nL_{j,P^l}}=0.
    \end{align}
    Since $k^l\rightarrow c=(0,\cdots,0),$ for each $m>0$, we can choose a number $N$ large enough such that $k_{j,P^l}\le\frac{1}{m}$ for any $1\leq j\leq n$ and $l\ge N$.
    By Lemma \ref{convex}, we know
    \[
    \pp{(\sum^{n}_{j=1}L_{j,P})}{k_i}>0.
    \]
    Therefore, 
$$\sum_{j=1}^{n}L_{i,P^l}\le n\frac{1}{\sqrt{m^2-1}}\bigg(-\ln\bigg(\sqrt{1+\frac{\tan^2{\frac{\pi}{n}}}{m^2}}-1\bigg)+D\bigg).$$ Thus \eqref{limsup} holds.
\end{proof}

\begin{lemma}\label{limita1}
    If $c_j=0$ for some $j$ $(1\leq j\leq n)$, then 
$\lim_{m\rightarrow+\infty}L^m_{j,P}=0$.
\end{lemma}
\begin{proof}
 We can choose an another sequence  $\widetilde{k}^{m}\rightarrow \textbf{0}=(0,\cdots,0)$ such that $\widetilde{k}_{j}^{m}=k_{j}^{m}$ and $\widetilde{k}_{i}^{m}<k_{i}^{m}$.
 By $\tilde{l}_{i,P^m}$ and $\tilde{L}_{i,P^m}$we denote the corresponding quantities of generalized circle packings on $P$ with geodesic curvatures $\{\tilde{k}^m\}$.
 Since 
$\pp{L_{j,P}}{k_i}<0 \text{ for } i\neq j$, we know $\widetilde{L}_{j,P}>L_{j,P}$. Owing to Lemma $3.1$, we have $\lim_{m\rightarrow+\infty}\widetilde{L}^{m}_{j,P}=0$. 
   
\end{proof}
 
\begin{lemma}\label{limita2}
 Let $I\subset \{1,\cdots,n\}$ be the a nonempty vertex set that the corresponding circle degenerate to a point, i.e.
\[
I=\{i:c_i=\infty\}.
\]
Then
    \begin{align}
     \lim_{m\rightarrow+\infty}\sum_{i\in I}L_{i,P^m}&=|I|\pi,~~\text{if}~~|I|\le n-3,\label{limit2} 
    \\\lim_{m\rightarrow+\infty}\sum_{i\in I}L_{i,P^m}&=(n-2)\pi,~~\text{if}~~n-3<|I|\le n.\label{limit3}
     \end{align} 

\end{lemma}
\begin{proof}
First, we prove the case when $|I|\le n-3.$

 We first show that geodesic curvatures $\{k_{P^m}\}_{m=1}^{\infty}$ of the dual disks $\{C_{P^m}\}_{m=1}^{\infty}$ must have an upper bound $K<+\infty.$ Otherwise, owing to Lemma \ref{vari_angle},  we have $\theta_{j,P^m}\rightarrow\pi$ for each $j\notin I$. Therefore, $$\lim_{m\to+\infty}\sum_{j=1}^n\theta_{j,P^m}=(n-|I|)\pi>2\pi,$$
which gives a contradiction.
Then by \eqref{angle_compute}, we have
\[\lim_{m\to+\infty}L_{i,P}^m=\frac{k_i^m}{\sqrt{(k_i^m)^2-1}}\alpha_{i,P^m}=\pi~~\forall i\in I.\]
Thus we have the equality (\ref{limit2}).

Now we turn to prove \eqref{limit3}. By Lemma \ref{gauss}, we have 
\[
\sum_{i\in I}L_{i,P^m}<(n-2)\pi.
\]
Therefore, we only need to prove
\begin{align}\label{contra}
    \liminf_{m\to+\infty}\sum_{i\in I}L_{i,P^m}=(n-2)\pi.
\end{align}

If not, there would exist a sub-sequence $\{k^{m_t}\}_{t=1}^{\infty}$ of $\{k^m\}_{m=1}^{\infty}$ such that
\begin{align}\label{assumption}
\lim_{t\to+\infty}\sum_{i\in I}L_{i,P^{m_t}}=A<(n-2)\pi,
\end{align}
for some non-negative number $A$. 
Then we claim that geodesic curvatures $\{k_{P^{m_t}}\}_{t=1}^{\infty}$ of the dual circles $C_{P^{m_t}}$ must tend to $+\infty$.

Otherwise, there is a sub-sequence $\{k^{m_{t_s}}\}$ of $\{k^{m_{t}}\}$ such that corresponding geodesic curvatures of dual circles is bounded above from $+\infty.$

Then by the same argument of the first case for $|I|\le n-3$, we have
\[
L_{i,P^{m_{t_s}}}\rightarrow\pi,~~\forall i\in I.
\]
Therefore
\[\lim_{t\to+\infty}\sum_{i\in I}L_{i,P^{m_{t_s}}}=|I|\pi\ge(n-2)\pi,\]
which contradicts to assumption \eqref{assumption}. Hence the claim is proved.

Let $\Omega_{P^{m_t}}$ be the domain defined as in Lemma \ref{area_change}. Since $\{k_{P^{m_t}}\}_{t=1}^{\infty}$ tends to infinity, the area of $\Omega_{P^{m_t}}$ must collapse to $0$. 
For $j\notin I$, it is easy to see that 
\[
\lim_{t\to+\infty}L_{j,P^{m_t}}=0.
\]
Therefore, by Lemma \ref{gauss}, we have
\[\lim_{t\to+\infty}\sum^n_{i=1 }L_{i,P^{m_t}}=\lim_{t\to+\infty}\sum_{i\in I}L_{i,P^{m_t}}=(n-2)\pi,\]
which contradicts to \eqref{assumption}.
\end{proof}

    By Lemma \ref{limita2}, we can  deduce the following corollary.
\begin{corollary}\label{cor1}
Let $\{C_1,\cdots, C_n\}$ be a generalized circle packing on an abstract polygon $P$ with $n$ vertices. Then
$L_{i,P}<\pi$ for each $i$.
\begin{proof}
    The reason is that, fixing $k_j$ for $j\neq i$, we have 
    \[
    \pp{L_{i,P}}{k_i}>0,
    ~~\lim_{k_i\rightarrow+\infty}L_{i,P}=\pi.\] 
\end{proof}
\end{corollary}

\subsection{The proof of Theorem \ref{main1}}
The aim of this section is to obtain Theorem \ref{main1} by the  variational principle. Since $L_i$ is the total geodesic curvature of the circle on vertex $i$, we have
\begin{align*}
    L_i=\sum_{P\in F_{\{i\}}}L_{i,P},
\end{align*}
where $F_{\{i\}}$ is the subset of $F_{g,n}$ consisting of $P$ with the vertex $i$.
We can define the potential function
\begin{equation}\label{W}
\mathcal{E}(s)=\sum_{P\in F_{g,n}} \mathcal{E}_{P}(s)\end{equation}
on $\mathbb{R}^{\vert V_{g,n}\vert}$, where $s_i=\ln{k}_i$ is defined in section \ref{section2}, and $\mathcal{E}_P$ only depends on $\{s_i\}_{i\in V(P)}$. We know that $\mathcal{E}_{P}(s)$ is strictly convex by Lemma \ref{convex}, therefore,  $\mathcal{E}(s)$ is strictly convex. Note that
\[\frac{\partial \mathcal{E}(s)}{\partial s_i}=L_{i}, \quad i\in V_{g,n}.\]

It is easy to see that $\nabla\mathcal{E} =(L_1,\cdots,L_{\vert V_{g,n}\vert})^T$, therefore the Hessian of $\mathcal{E}$ equals to a Jacobi matrix, i.e.
\[\widetilde{M} = \frac{\partial(L_1,...,L_{\vert V_{g,n}\vert})}{\partial(s_1,...,s_{\vert V_{g,n}\vert})}=\begin{pmatrix}
 	 		\frac{\partial L_1}{\partial s_1}&\cdots&\frac{\partial L_{1}}{\partial s_{\vert V_{g,n}\vert}}\\
 	 		\vdots&\ddots&\vdots\\
 	 		\frac{\partial L_{\vert V_{g,n}\vert}}{\partial s_{1}}&\cdots&\frac{\partial L_{\vert V_{g,n}\vert}}{\partial s_{\vert V_{g,n}\vert}}\\
\end{pmatrix}.\]
\begin{lemma}\label{pro 3.1}
    The Jacobi matrix $\widetilde{M}$ is positive definite.
\end{lemma}
\begin{proof}
    By the proof of lemma \ref{convex}, we have
$$
\frac{\partial L_i}{\partial s_i}=\frac{\partial(\sum_{P\in F_{\{i\}}}L_{i,P})}{\partial s_i}>0,~~\forall i\in V_{g,n}.
$$
and
$$
\frac{\partial L_j}{\partial s_{i}}=\frac{\partial(\sum_{P\in F_{\{j\}}}L_{j,P})}{\partial s_{i}}\le0,~~\forall i\ne j\in V_{g,n}.
$$
By Lemma \ref{area_change} and Lemma \ref{convex}, we obtain
$$
\begin{aligned}
\big|\frac{\partial L_i}{\partial s_i}\big|-\sum_{j\ne i}\big|\frac{\partial L_j}{\partial s_i}\big|&=\frac{\partial L_i}{\partial s_i}+\sum_{j\ne i}\frac{\partial L_j}{\partial s_i}=\sum_{P\in F_{g,n}}\frac{\partial(\sum_{j\in V(P)} L_{j,P})}{\partial s_i}\\
&=-\sum_{P\in F_{g,n}}\frac{\operatorname{Area}(\Omega_P)}{\partial s_i}>0.
\end{aligned}
$$
 The Jacobi matrix $\widetilde{M}$ is a strictly diagonally dominant matrix with positive diagonal entries, therefore, $\widetilde{M}$ is positive definite.
\end{proof}
The following property of convex functions plays a key role  in the proof of Theorem \ref{main1}.

\begin{lemma}\cite{Dai}\label{injective}
Let $\mathcal{T}:\mathbb{R}^n\to\mathbb{R}$ be a $C^2$-smooth strictly convex function with apositive definite Hessian matrix. Then its gradient $\nabla \mathcal{T}:\mathbb{R}^n\to\mathbb{R}^n$ is a smooth embedding.
\end{lemma}

Now we give the proof of Theorem \ref{main1} as follows.
\begin{proof}[Proof of Theorem \ref{main1}] We only need to prove that $\nabla \mathcal{E}(\mathbb{R}_{+}^{|V_{g,n}|})$ is equal to \[\mathcal{L}=\left\{(L_1,\cdots,L_{\vert V_{g,n}\vert})\in\mathbb{R}^{\vert V_{g,n}\vert}_{+}\left\vert\right.  \sum_{w\in W}L_w<\sum_{P\in F_W}\pi \min\{N(P,W),N(P)-2\},~~\text{for each}\ W\subset
 V_{g,n}\right\}.\]
 By Lemma \ref{gauss} and Corollary \ref{cor1}, we have  $\nabla\mathcal{E}(\mathbb{R}^{|V|_{g,n}})\subset\mathcal{L}$. Thanks to Lemma \ref{injective}, $\nabla\mathcal{E}$ is injective. 
By Brouwer's Theorem on the Invariance of Domain, we need to analyze the boundary of $\nabla \mathcal{E}(\mathbb{R}_{+}^{|V_{g,n}|})$. Taking a sequence $\{s^{m}\}_{m=1}^{\infty}$ such that
\[\lim_{m\to\infty} s^{m}=a\in[-\infty,+\infty]^{\vert V_{g,n}\vert},\]
where $a_i=-\infty$ or $+\infty$ for some $i\in V$, we need to prove that $\nabla \mathcal{E}(s^{m})$ converges to the boundary of  $\mathcal{L}$.

Set  $ W,W'\subset
 V_{g,n}$ such that $a_i=+\infty$ (resp. $a_i=-\infty$) for $i\in W$ (resp. $i\in W'$). Here $F_W$ is the set of faces having at least one vertex in $W$. Suppose that  $P\in F_W$. We divide the proof into two cases as follows.
\begin{itemize}
    \item[($a$)]  If the vertices of $W$ in some  face $P\in F_{W}$ are less than $n-2$, then by (\ref{limit2}) of Lemma \ref{limita2},
    we have $$\lim_{m\to +\infty}\sum_{i\in W\cap V(P)}L_{i,P}(s^m)=N(P,W)\pi. $$
 \item[($b$)] If the vertices of $W$ in some  face $P\in F_{W}$ are more than $n-3$, then by (\ref{limit3}) of Lemma \ref{limita2},
    we obtain $$\lim_{m\to +\infty}\sum_{i\in W\cap V(P)}L_{i,P}(s^m)=(N(P)-2)\pi. $$
\end{itemize}
By $(a)$ and $(b)$, we have $$
\lim_{m\rightarrow+\infty}\sum_{i\in W}L_i(s^m)=\sum_{P\in F_W}\pi \min\{N(P,W),N(P)-2\}.
$$

Suppose $i\in W'$. Since $s_i\rightarrow\-\infty$, by Lemma  \ref{limita1}, we have
$$\lim_{m\to +\infty}L_i(s^m)=\lim_{m\rightarrow+\infty}\sum_{P\in F_{\{i\}}}L_{i,P}(s^m)=0. $$
Hence the proof of this lemma is complete.
\end{proof}
\section{Combinatorial Ricci flow}By the change of variables $s_i=\ln{k_i}$,  the flow (\ref{ricci_flow}) is also written as
\begin{equation}\label{Flow}
\frac{ds_i}{dt}=-(L_i-\hat{L}_i)~~,\forall i\in V_{g,n},
\end{equation}
 which is the negative gradient flow of the potential function $\mathcal{E}_{\hat{L}}=\mathcal{E}-\sum_i\int\hat{L}_i\mathrm{d}s_i.$
\begin{lemma}\label{lem4.1}
For any initial value $k(0)$, the solution of the flow (\ref{Flow}) exists for all time $t\in(0,\infty]$. Moreover, it is unique.
\end{lemma}
\begin{proof}
Note that $L_i$ is a $C^1$ function of $(s_1,\cdots,s_n)$. By Cauchy-Lipschitz Theorem, the flow (\ref{Flow}) exists a unique solution $s(t)$ for some $t\in[0,\varepsilon)$. Since corollary \ref{cor1} indicates that
\[\left\vert L_{i}\right\vert<c_i\pi,\]
where $c_i=|F_{\{i\}}|$, we have
\[\left\vert\frac{ds_{i}}{dt}\right\vert
=\left\vert L_{i}-\hat{L}_{i}\right\vert
\leq\left\vert L_{i}\right\vert+\left\vert\hat{L}_{i}\right\vert
\leq c_i\pi+\left\vert\hat{L}_{i}\right\vert.\]
Hence, the flow \eqref{Flow} exists for all $t\in [0,\infty)$. 
The uniqueness of the flow can be easily deduced from the classical theory of ordinary differential equations.
\end{proof}
Since the flow (\ref{Flow}) is the negative gradient flow of $\mathcal{E}_{\hat{L}}(s)$,  $\mathcal{E}_{\hat{L}}$ exists a unique critical point if and only if there exists $s^*\in \mathbb{R}^{|V_{g,n}|}$ such that $\nabla\mathcal{E}_{\hat{L}}(s^*)=\hat{L},$ which is equivalent to $\hat{L}\in \mathcal{L}$ by Theorem \ref{main1}. Then we have the following lemma.
\begin{lemma}\label{lem4.2}
$\mathcal{E}_{\hat{L}}(s)$ exists a unique critical point if and only if $\hat{L}\in\mathcal{L}$.
\end{lemma}
The following property of convex functions plays an important role in proving  Theorem \ref{flowthm}, a proof of which could be found in \cite[Lemma 4.6]{Gexu}.
\begin{lemma}\label{convex_property}
Suppose $F(x)$ is a $C^{1}-$smooth convex function on $\mathbb{R}^{n}$ with $\nabla F(x_{0})=0$ for some $x_{0}\in\mathbb{R}^{n}$. Suppose $F(x)$ is $C^{2}-$smooth and strictly convex in a neighborhood of $x_0$. Then the following statements hold:
\begin{itemize}
\item[($a$)] $\nabla F(x)\neq0$ for any $x\neq x_0$.
\item[$(b)$] Then $\lim_{\Vert x\Vert\to+\infty}F(x)=+\infty$.
\end{itemize}
\end{lemma}

Now we begin to prove Theorem \ref{flowthm}.
\begin{proof}[Proof of Theorem \ref{flowthm}]
    We begin with the ``if''  part. Assume that $\{s(t)\}_{t\ge 0}$ is a solution of the flow \eqref{Flow}. Since $\hat{L}\in\mathcal{L}$,  $\mathcal{E}_{\hat{L}}$ has a unique critical point  by Theorem \ref{main1}. Since $\mathcal{E}_{\hat{L}}$ is strictly convex, it has a global minimum point.
    By the property of negative gradient flow, $\mathcal{E}_{\hat{L}}(s(t))$ is bounded from above.
Since $\mathcal{E}_{\hat{L}}$ is proper, $\{s(t)\}_{t\ge 0}$ is contained in a compact set in $\mathbb{R}^n$. 
    By \eqref{Flow}, we have
    \[
    \ddt{\|L(s(t))-\hat{L}\|^2}=-2(L-\hat{L})^t\nabla_sL(L-\hat{L}).
    \]
    where $\nabla_sL$ is the Hessian of $\mathcal{E}$ and is positive definite. Since $\{s(t)\}_{t\ge 0}$ lies in a compact region, the minimal eigenvalue of $\nabla_sL$ must be larger than a positive constant $\lambda_0$ as $t\ge 0.$ Hence we have 
    \[
   \|L(s(t))-\hat{L}\|^2\le exp(-2\lambda_0t)\|L(s(0))-\hat{L}\|^2.
    \]
    Since 
    the right side of \eqref{Flow} converges exponentially, $\{s(t)\}_{t\ge 0}$ converges exponentially fast to the logarithm of geodesic curvatures of the desired generalized circle packing.

    As for the ``only if'' part, suppose that $s(t)$ converges to some $s^*$,
    then
    \[\mathcal{E}_{\hat{L}}(s(i))\rightarrow \mathcal{E}_{\hat{L}}(s^*)~~\text{as} ~i\rightarrow+\infty.\]
    Therefore, by the mean value theorem,
    \[\mathcal{E}_{\hat{L}}(i+1)-\mathcal{E}_{\hat{L}}(i)=\ddt{\mathcal{E}_{\hat{L}}(s(\eta_i))}\rightarrow 0~~\text{as}~i\rightarrow\infty,\]
    for some $\eta_i\in[i,i+1].$
    Since
    \[
    \ddt{\mathcal{E}_{\hat{L}}(s(t))}=-\|L-\hat{L}\|^2
    ,\]
    we have
    \[
   \|L(s(\eta_i))-\hat{L}\|\rightarrow 0~~\text{as}~i\rightarrow\infty.
    \]
    There exists  $s^*$ such that $L(s^*)-\hat{L}=0$. Hence $\mathcal{E}_{\hat{L}}$ has a critical point. By Theorem \ref{main1},  we finally finish the proof.
\end{proof}

\section{Acknowledgments}
Guangming Hu is supported by NSF of China (No. 12101275). Yi Qi is supported by NSF of China (No. 12271017). Puchun Zhou is supported by Shanghai Science and Technology Program [Project No. 22JC1400100]. The authors would like to thank Xin Nie for helpful discussions.
 
\Addresses

\begin{thebibliography}{99}
\bibitem{andreev1} E. M. Andreev, \emph{Convex polyhedra in Lobachevsky spaces.} (Russian) Mat. Sb. (N.S.) 81 (123),
 1970, 445--478.
\bibitem{andreev2} E. M. Andreev, \emph{On convex polyhedra of finite volume in Lobacevskii space,} Math. USSR Sbornik, 12(3), 1971, 225-259.
\bibitem{BHS}T. Ba, G. Hu and Y. Sun, \emph{Circle packings and total geodesic curvatures in hyperbolic background geometry,} preprint, \url{https://arxiv.org/abs/2307.13572.}

\bibitem{bo} A. I. Bobenko, B. A. Springborn, \emph{Variational principles for circle patterns and Koebe's theorem,} Trans. Amer. Math. Soc. 356, 2004, 659--689.

\bibitem{chow-Luo} B. Chow and F. Luo, \emph{Combinatorial Ricci flows on surfaces}, J. Differential Geom. 63, 2003, 97--129.
\bibitem{colin} Y. Colin de Verdiere, \emph{Un principe variationnel pour les empilements de cercles}, Invent.
Math. 104, 1991, 655--669.
\bibitem{Dai} J. Dai, X. D. Gu and F. Luo, \emph{Variational principles for discrete surfaces}, Advanced Lectures in Mathematics 4, Higher Education Press, Beijing, 2008.
\bibitem{Feng} 
K. Feng, H. Ge, and B. Hua, \emph{Combinatorial Ricci flows and the hyperbolization of a class of compact 3-manifolds,} Geom, Topol., 26(3), 2022, 1349--1384.
\bibitem{Gehua} H. Ge, \emph{Combinatorial calabi flows on surfaces,} Trans.  Amer.
Math. Soc., 370(2), 2018, 1377--1391.
\bibitem{ge2} H. Ge, B. Hua, Z. Zhou, \emph{Combinatorial Ricci flows for ideal circle patterns}, Adv. Math. 383 (2021), Paper No. 107698.
\bibitem{GeJiang}
H. Ge, W. Jiang, and L. Shen, \emph{On the deformation of ball packings,} Adv.  Math., 398(2022),  No. 108192. 

\bibitem{Gexu}H. Ge and X. Xu, \emph{On a combinatorial curvature for surfaces with inversive distance circle packing
metrics,} J. Funct. Anal. 275(3), 2018, 523--558.

\bibitem{David} D. Glickenstein. \emph{A combinatorial Yamabe flow in three dimensions,} Topology,
44(4), 2005, 791--808.
\bibitem{David2} D. Glickenstein,\emph{ Discrete conformal variations and scalar curvature on piecewise flat two and three dimensional manifolds,} J. Differential Geom. 87 (2), 2011, 201--237.
\bibitem{David3} D. Glickenstein and  J. Thomas, \emph{Duality structures and discrete conformal variations of piecewise
constant curvature surfaces,} Adv. Math. 320, 2017, 250--278.



\bibitem{Gu1} X. Gu, F. Luo, J. Sun and T. Wu. \emph{A discrete uniformization
theorem for polyhedral surfaces,} J. Differential Geom., 109(2), 2018, 223--256.

\bibitem{Gu2}X. Gu, R. Guo, F. Luo, J. Sun and T. Wu. \emph{A discrete uniformization theorem for polyhedral surfaces II,} J. Differential Geom., 109(3), 2018, 431--466.

\bibitem{Gu3}X. Gu  and S. Yau, \emph{Computational conformal geometry,} Volume I. International Press Somerville, MA, 2008.
\bibitem{guo-luo} R. Guo, F. Luo, \emph{Rigidity of polyhedral surfaces. II}, Geom. Topol. 13 (3), 2009, 1265--1312.
\bibitem{Hamilton}R. S. Hamilton, \emph{Three-manifolds with positive Ricci curvature,} J.  Differential Geom., 17(2), 1982, 255--306.
\bibitem{Koebe} P. Koebe, \emph{Kontaktprobleme der konformen Abbildung,} Abh. S$\ddot{a}$chs. Akad. Wiss. Leipzig
Math.-Natur. 88, 1936, 141-164.
\bibitem{Leibon}G. Leibon,\emph{Characterizing the Delaunay decompositions of compact hyperbolic surfaces,} Geom. Topol., 6, 2002, 361--391.



 \bibitem{Luo11}F. Luo, \emph{Combinatorial Yamabe flow on surfaces,} Commun. Contemp. Math., 6 (5), 2004
765--780.
\bibitem{Luo12}F. Luo, \emph{A characterization of spherical polyhedral surfaces,} J. Differential Geom., 74, 2006, 407--424.
\bibitem{Luo13}F. Luo, \emph{On Teichm\"uller spaces of surfaces with boundary,} Duke Math. J., 139, 2007,
463--482.

\bibitem{Luo}F. Luo. \emph{Rigidity of polyhedral surfaces I,} J. Differential Geom., 96(2), 2014, 241--302.

\bibitem{Luo2} F. Luo, \emph{ Rigidity of polyhedral surfaces, III,} Geom. Topol.,  15, 2011, 2299--2319.

\bibitem{nie} X. Nie, \emph{On circle patterns and spherical conical metrics}, preprint, \url{https://arxiv.org/abs/2301.09585}.
\bibitem{Perelman1} G. Perelman, \emph{The entropy formula for the Ricci flow and its geometric applications,} arXiv: 0211159.
\bibitem{Perelman2} G. Perelman, \emph{Ricci flow with surgery on three-manifolds,} arXiv: math.DG/0303109.

\bibitem{Perelman3} G.Perelman, \emph{Finite extinction time for the solutions to the Ricci flow on certain three manifolds,} arXiv: math.DG/0307245.


\bibitem{Rivin}I. Rivin, \emph{Euclidean structures on simplicial surfaces and hyperbolic volume,} Ann. of
Math.,  139(2), 1994, 553--580.
\bibitem{Schlenker} J-M. Schlenker,\emph{Circle patterns on singular surfaces,} Discrete Comput. Geom., 40,
2008, 47--102.


 \bibitem{thurston} W. P. Thurston, \emph{Geometry and topology of 3-manifolds,}
 Princeton lecture notes, 1976.
\end{thebibliography}
\end{document}